\documentclass[11pt]{amsart} 

\usepackage[all,cmtip]{xy}
\usepackage{amsmath, amssymb}
\usepackage{mathtools}
\usepackage[T1]{fontenc}
\usepackage{lmodern}
\usepackage[english]{babel}
\usepackage[hidelinks]{hyperref}
\hypersetup{pdftitle={The radial derivative on the graded M\"obius algebra},pdfauthor={Thomas Sinclair}}
\usepackage[marginpar=2.5cm]{geometry}
\usepackage{xcolor}\usepackage{marginnote}
\usepackage{orcidlink}

\newtheorem{thm}{Theorem}[section]
\newtheorem{prop}[thm]{Proposition}
\newtheorem{lem}[thm]{Lemma}
\newtheorem{cor}[thm]{Corollary}
\newtheorem{main}{Theorem}

\newtheorem{mconj}[main]{Conjecture}

\theoremstyle{definition}

\newtheorem{example}[thm]{Example}
\newtheorem{question}[thm]{Question}

\theoremstyle{remark}
\newtheorem{remark}[thm]{Remark}

\newcommand{\R}{\mathbb{R}}
\newcommand{\Q}{\mathbb{Q}}
\newcommand{\cc}{\mathcal}
\DeclareMathOperator{\rk}{rk}
\DeclareMathOperator{\cl}{cl}
\DeclareMathOperator{\Span}{span}
\newcommand{\del}{\partial}
\newcommand{\ip}[2]{\langle #1, #2\rangle}
\newcommand{\ab}[1]{\left| #1 \right|}

\title[The radial derivative]{The radial derivative on the graded M\"obius algebra}
\author[T.~Sinclair]{Thomas Sinclair \orcidlink{0000-0003-0401-7232}}

\address{Mathematics Department, Purdue University,\linebreak[1] 150 N. University Street, West Lafayette, IN 47907-2067}
\email{tsincla@purdue.edu}
\urladdr{http://www.math.purdue.edu/~tsincla/}

\date{\today}

\begin{document}

\begin{abstract}
Let $M$ be a simple matroid and let $B(M)$ be the graded M\"obius algebra of its lattice of flats.  The ordered-basis weights of flats define an inner product for which the adjoints of atom multiplication become ordinary coordinate derivatives under the basis-polynomial realization. From this, we construct a canonical global lowering operator $D_\beta$ which acts as ordinary differentiation on a canonical ``radial'' copy of a truncated polynomial algebra.  Allowing both $D_\beta$ and the coordinate derivatives to act produces a graded cyclic module with Hilbert series
\[
H_{\beta,M}(q)=\sum_{k=0}^r h_k^\beta(M)q^k.
\]
We give examples of matroids with the same Derksen $\mathcal G$-invariant and the same classical apolar Hilbert series but different $H_\beta$.  Hence $H_\beta$ cannot be the restriction to simple matroids of a valuative matroid invariant.

We conjecture that $H_\beta$ is log-concave and that the excess $H_\beta-H^0$ over the classical apolar Hilbert series has unimodal coefficients. Exact computations verify both properties for all $950$ simple matroids on eight elements and all $376{,}467$ on nine elements. For the generalized theta family containing Larson's counterexample to Whitney log-concavity, we compute the first four coefficients and prove the log-concavity inequality at the corresponding index.
\end{abstract}

\maketitle

\section{Introduction}

For a rank-$r$ matroid $M$, we write $L(M)$ for its lattice of flats and $L_k(M)$ for the flats of rank $k$. The \emph{Whitney numbers of the second kind} are defined as $W_k(M) := |L_k(M)|$.  That the Whitney numbers of the second kind are top-heavy,
\[
W_k(M)\le W_{r-k}(M),\qquad k\le r/2,
\]
is a landmark achievement of singular Hodge theory developed by Braden, Huh, Matherne, Proudfoot, and Wang \cite{BHMPW}.  Two stronger shape expectations have recently failed: Larson constructed graphic counterexamples to Mason's log-concavity conjecture \cite{Mason,Larson}, and flat-count sequences need not even be unimodal by work of Divoux, Larson, Lowen, and Wang \cite{DivouxLarsonLowenWang}.  This makes it natural to ask whether the lattice of flats carries a different canonical numerical profile with a more rigid shape.

We propose a candidate for such a profile derived from a differential calculus on the graded M\"obius algebra $B(M)$. The graded M\"obius algebra has a rich and extensively developed theory: see, for instance, \cite{BHMPW,LaClairMastroeniMcCulloughPeeva,Lee,MaenoNumata}. The source of the differential calculus considered here lies in an operator-theoretic approach to $B(M)$ from recent work of the author \cite{SinclairHamiltonian,SinclairFFC}.  Let $L=L(M)$ and let $e_x\in \R[L]$ denote the basis vector indexed by a flat $x$.  If $a$ is an atom, the linear operator on $\R[L]$ induced by multiplication by $e_a$ is denoted $L_a$.  Write $\beta_x$ for the number of ordered atom bases of $x$.  The vectors
\[
v_k :=\Bigl(\sum_aL_a\Bigr)^ke_{\hat 0}=\sum_{x\in L_k}\beta_xe_x
\]
form a canonical realization of the truncated polynomial ring $\R[t]/(t^{r+1})$, which we will refer to as the \emph{radial core} of $B(M)$.  Equip $\R[L]$ with the inner product
\[
\ip{e_x}{e_y}_\beta=\frac{\delta_{x,y}}{\beta_x},
\]
and let $L_a^\beta$ be the adjoint of $L_a$ with respect to this inner product, that is, \[\ip{L_ae_x}{e_y}_\beta = \ip{e_x}{L_a^\beta e_y}_\beta\] for all $x,y\in L$. Note that $L_a^\beta$ is degree-one lowering, taking $\R[L_k]$ into $\R[L_{k-1}]$.  As we will see in the sequel, under a natural basis-polynomial realization these adjoints are ordinary coordinate derivatives. We will define a canonical degree-one lowering operator $D_\beta$ with the property that
\[
D_\beta v_k=kv_{k-1}.
\]
Thus, $D_\beta$ is ordinary differentiation on the radial core. Away from that core it detects the nonuniform extension geometry of the lattice. While the operators $\{L_a^\beta\}_{a\in A}$ commute pairwise, $D_\beta$ generally does not commute with them.
We mention in passing that the operator $D_\beta$ arose from an attempt to understand differential formulas for finite free convolution in the context of \cite{SinclairFFC}.

Starting from $v_r$, let $\cc C^\beta(M)=(\cc C_k^\beta(M))_{k=0}^r$ be the cyclic module generated by $D_\beta$ and the $L_a^\beta$, graded by word length. Equivalently,
\[
\cc C_0^\beta(M):=\R v_r,\qquad
\cc C_{k+1}^\beta(M):=D_\beta\cc C_k^\beta(M)+\sum_{a\in A}L_a^\beta\cc C_k^\beta(M).
\]

We set
\[
h_k^\beta(M) :=\dim \cc C_k^\beta(M),\qquad
H_{\beta,M}(q) :=\sum_{k=0}^rh_k^\beta(M)q^k.
\]
If $D_\beta$ is omitted, one recovers the Macaulay inverse system of the basis-generating polynomial $B_M$. We denote its Hilbert function by $h_k^0(M)$.  The first comparison is immediate but useful:
\[
\quad h_k^0(M)\le h_k^\beta(M)\le W_{r-k}(M).
\]
The first inequality exhibits $H_\beta$ as a canonical enlargement of a symmetric Gorenstein Hilbert function.  The second records that differential degree $k$ lies in the rank-$(r-k)$ flat layer.  The reverse comparison $W_k\le h_k^\beta$ does not hold; it fails for $938$ of the $950$ simple eight-element matroids.

\medskip

This note collects preliminary results and tested conjectures around this invariant.

\medskip

First, the invariant is not merely a repackaging of standard matroid data.

\begin{main}[Separation beyond the $\mathcal G$-invariant]\label{thm:G-separation}
There exist simple rank-five matroids $P,Q$ with
\[
\mathcal G(P)=\mathcal G(Q),\qquad H_P^0(q)=H_Q^0(q),
\]
but
\[
H_{\beta,P}(q)=1+9q+30q^2+25q^3+8q^4+q^5,
\]
\[
H_{\beta,Q}(q)=1+9q+28q^2+25q^3+8q^4+q^5.
\]
In both cases
\[
H_P^0(q)=H_Q^0(q)=1+8q+25q^2+25q^3+8q^4+q^5.
\]
\end{main}

Since the $\mathcal G$-invariant introduced by Derksen \cite{Derksen} is universal among valuative matroid invariants by Derksen and Fink \cite{DerksenFink}, $H_\beta$ cannot be the restriction to simple matroids of a valuative matroid invariant.

\medskip

Second, the exact computations in Section~\ref{sec:computation} suggest two shape properties. Put
\[
\delta_k(M) :=h_k^\beta(M)-h_k^0(M).
\]

\begin{mconj}[Log-concavity]\label{conj:lc}
For every simple matroid $M$,
\[
(h_k^\beta)^2\ge h_{k-1}^\beta h_{k+1}^\beta,
\qquad 1\le k\le r-1.
\]
\end{mconj}

\medskip

\begin{mconj}[Unimodality of the defect]\label{conj:defect-unimodal}
For every simple matroid $M$, the sequence $(\delta_k(M))_{k=0}^r$ is unimodal.
\end{mconj}

Example~\ref{ex:defect-lc} shows that the defect can fail log-concavity.

\medskip

\begin{remark}
The ``front-loading'' conjecture described in the first version of this manuscript asserted the inequalities
\[
\delta_k(M)\ge\delta_{r-k}(M),\qquad 0\le k\le r/2.
\]
By apolar symmetry this is equivalent to $h_k^\beta(M)\ge h_{r-k}^\beta(M)$ over the same range.
Guo, Mo, Li, and Deng \cite[\S J.2]{Guo-etal} refuted this conjecture with the simple rank-five matroid \texttt{M195313} on nine elements. The exact recomputation in differential degree gives
\[
h^0=(1,9,27,27,9,1),\qquad h^\beta=(1,9,31,32,9,1).
\]
Here $\delta_2=4<5=\delta_3$. This matroid is the sole failure of front-loading in the complete nine-element census. Levinstein \cite{Levinstein} found a graphic counterexample on ten elements, obtained by adding one edge within each bipartition class of $K_{2,4}$. Its weighted Hilbert vector is $(1,11,31,32,10,1)$. Further graphical counterexamples were obtained in work of Chahal and Gupta \cite{Chahal}.
\end{remark}

\medskip

Third, the recent Whitney counterexamples give a particularly pointed test.  Let $M_t$ be the cycle matroid of the generalized theta graph with four internally disjoint paths of lengths $(1,t,t,t)$.  Larson's example is $M_{26}$.  We prove:

\begin{main}[Generalized theta family]\label{thm:theta}
For $t\ge6$,
\[
(h_0^\beta,h_1^\beta,h_2^\beta,h_3^\beta)
=
\left(
1,
3t+1,
\frac{9t^2+9t-6}{2},
\frac{9t^3+21t^2-28t}{2}
\right).
\]
Consequently,
\[
(h_2^\beta)^2-h_1^\beta h_3^\beta
=
\frac{27t^4+18t^3+99t^2-52t+36}{4}>0.
\]
For $t=26$ the first four coefficients are $(1,79,3156,85826)$ and the critical gap is $3{,}180{,}082$.
\end{main}

Thus the $H_\beta$ inequality holds at the differential index corresponding to Larson's counterexample to Mason's log-concavity conjecture. The two conjectures above hold in the complete eight- and nine-element censuses. Section~\ref{sec:questions} asks what structure could account for these inequalities.


\section{Differential Calculus and the Apolar Module}\label{sec:construction}

Let $L$ be a finite geometric lattice of rank $r$ with atom set $A$. We will denote the bottom and top elements of $L$ by $\hat 0$ and $\hat 1$, respectively. For instance, we could take $L = L(M)$ for $M$ a simple matroid. Its graded M\"obius algebra $B(L)$ is the vector space $\R[L]$ equipped with the product
\[
e_x\diamond e_y :=
\begin{cases}
e_{x\vee y},&\rk(x)+\rk(y)=\rk(x\vee y),\\
0,&\text{otherwise}.
\end{cases}
\]
If $L = L(M)$ we will alternately write $B(M)$ for $B(L(M))$.
For $a\in A$, let $L_a$ be multiplication by $e_a$ and $U=\sum_aL_a$.  The $L_a$ commute and satisfy $L_a^2=0$.
The graded M\"obius algebra is standard in current matroid Hodge theory: it appears in the proof of top-heaviness \cite{BHMPW}, its homological properties have recently been studied in \cite{LaClairMastroeniMcCulloughPeeva}, and Lee gives a tropical-cohomological realization for arbitrary matroids \cite{Lee}.  

\medskip

Our new ingredients are the weighted adjoint and differential structures placed on this algebra.

\medskip

We begin by constructing the radial core of $B(L)$. Let
\[
u:=\sum_{a\in A}e_a\in B(L)_1
\]
and, for $0\le k\le r$, let
\[
v_k:=u^k
=\Bigl(\sum_{a\in A}L_a\Bigr)^ke_{\hat 0} = \sum_{x\in L_k} \beta_x e_x.
\]
The identity is a standard expansion using that $\{L_a\}_{a\in A}$ is a commuting family and that $L_a^2=0$. Indeed, expanding $u^k$ gives
\[
u^k=\sum_{(a_1,\ldots,a_k)\in A^k}
e_{a_1}\cdots e_{a_k}.
\]
A product $e_{a_1}\diamond \dotsb \diamond e_{a_k}$ is nonzero precisely when
$a_1,\ldots,a_k$ are independent, in which case it equals
$e_x$ for $x={\cl\{a_1,\ldots,a_k\}}$.  Hence the coefficient of $e_x$, for
$x\in L_k$, is exactly the number $\beta_x$ of ordered atom bases
of $x$.

\begin{prop}[The radial core]\label{prop:radial-core}
The subspace
\[
\cc R(L):=\Span\{v_0,\ldots,v_r\}\subseteq B(L)
\]
is a graded subalgebra, and the map
\[
\R[t]/(t^{r+1})\longrightarrow \cc R(L),
\qquad
t^k\longmapsto v_k,
\]
is an isomorphism of graded algebras.
\end{prop}

\begin{proof}
Since $e_{\hat 0}$ is the identity of the graded M\"obius algebra and
$L_a$ is multiplication by $e_a$, the operator $U$
is the multiplication operator associated to $u$; consequently,
\[
v_k=U^ke_{\hat 0}=u^k.
\]
Associativity therefore gives
\[
v_iv_j=u^{i+j}=v_{i+j}
\]
whenever $i+j\le r$.  If $i+j>r$, then $u^{i+j}=0$, since the graded
M\"obius algebra is concentrated in degrees $0,\ldots,r$.

For each $0\le k\le r$, the vector $v_k$ is nonzero and belongs to the
homogeneous component $B(L)_k$.  Since distinct $v_k$ lie in distinct
graded components, they are linearly independent.  Hence
$\cc R(L)$ has basis $v_0,\ldots,v_r$, and the homomorphism
\[
\R[t]/(t^{r+1})\to\cc R(L),\qquad t\mapsto v_1=u,
\]
is bijective. \qedhere
\end{proof}

\medskip

For $x\in L$, consider the multiaffine polynomial
\[
B_x(z) :=\sum_{I\textup{ a basis of }x}z^I,
\]
where $z = (z_a)_{a\in A}$ and for $S\subseteq A$ $z^S := \prod_{a\in S} z_a$.

For $n=|A|$, let $\cc A_n:=\R[z_a:a\in A]/(z_a^2:a\in A)$ be the squarefree algebra of multiaffine polynomials in $(z_a)_{a\in A}$, as in \cite{SinclairFFC}. We define the linear map $\Phi_\beta: \R[L]\to \cc A_n$ by
\[
\Phi_\beta(e_x)=\frac{B_x(z)}{\beta_x}.
\]

\begin{prop}\label{prop:adjoint-polynomial}
For $y\in L$,
\[
L_a^\beta e_y=
\sum_{\substack{x\lessdot y\\x\vee a=y}}
\frac{\beta_x}{\beta_y}e_x.
\]
The map $\Phi_\beta$ is injective and intertwines $L_a^\beta$ with $\del_a$, that is $\Phi_\beta(L_a^\beta e_y) = \del_a\Phi_\beta(e_y)$.
\end{prop}

\begin{proof}
The adjoint formula follows by pairing $L_ae_x$ with $e_y$: see \cite{SinclairHamiltonian} for details.  The supports of the $B_x$ are disjoint because an independent set has a unique closure, so $\Phi_\beta$ is injective.  Differentiating $B_y$ by $z_a$ deletes $a$ from each basis of $y$ containing it; the remaining basis has closure $x\lessdot y$ with $x\vee a=y$, giving the displayed intertwining relation. \qedhere
\end{proof}

Set
\[
U^\beta:=U^{*_{\beta}}=\sum_{a\in A}L_a^\beta,
\qquad
Ke_x=\rk(x)e_x,
\qquad
Se_x=s(x)e_x,
\]
where $s(x)=\ab{A}-\ab{\{a:a\le x\}}$. Define 
\[
S^{-1}e_x:=\begin{cases}
s(x)^{-1}e_x,&x<\hat 1,\\
0,&x=\hat 1.
\end{cases}
\]

\begin{prop}\label{prop:D}
The operator
\[
D_\beta=S^{-1}U^\beta K
\]
satisfies
\[
D_\beta v_0=0,\qquad D_\beta v_k=kv_{k-1}\quad(1\le k\le r).
\]
Under the isomorphism of Proposition~\ref{prop:radial-core}, $U$ restricts to multiplication by $t$ and $D_\beta$ restricts to $d/dt$ on the radial core.
\end{prop}

\begin{proof}
Fix $1\le k\le r$ and $x\in L_{k-1}$. For each atom $a\nleq x$, the flat $y=x\vee a$ covers $x$, and the $e_x$-coefficient of $L_a^\beta(\beta_y e_y)$ is $\beta_x$. Since there are $s(x)$ atoms not below $x$,
\[
U^\beta v_k=Sv_{k-1},
\]
and therefore $D_\beta v_k=kS^{-1}Sv_{k-1}=kv_{k-1}$. Also $D_\beta v_0=0$ because $Kv_0=0$.
\end{proof}

For a cover $x\lessdot y$, put
\[
m(x,y)=\#\{a\in A:a\nleq x,\ x\vee a=y\}.
\]
Thus $s(x)=\sum_{y\gtrdot x}m(x,y)$.  The following gives an intrinsic description of the extension away from the radial subspace.

\begin{prop}\label{prop:naturality}
Put $g_x=\beta_xe_x$.  For every $y\in L$ and atom $a\in A$,
\begin{equation}\label{eq:rescaled}
L_a^\beta g_y=
\sum_{\substack{x\lessdot y\\x\vee a=y}}g_x,
\qquad
D_\beta g_y
=\rk(y)\sum_{x\lessdot y}\frac{m(x,y)}{s(x)}g_x.
\end{equation}
\end{prop}

\begin{proof}
The first identity is the adjoint formula after multiplying by $\beta_y$.  Summing it over the atoms and applying $S^{-1}K$ gives the formula for $D_\beta$.  The atoms not below $x$ are partitioned by the cover $x\vee a$, proving
$\sum_{y\gtrdot x}m(x,y)=s(x)$. \qedhere
\end{proof}

Thus, for fixed $x<\hat 1$, the coefficients
\[
p(x,y)=\frac{m(x,y)}{s(x)},\qquad y\gtrdot x,
\]
form the transition probabilities of the canonical random extension given by uniformly choosing an atom not below $x$ and replacing $x$ by $x\vee a$.  That is, apart from the factor $\rk(y)$, $D_\beta$ is the transpose of this upward Markov kernel.

\begin{prop}[Naturality and diagonal normalization]\label{prop:naturality-2}
If $\varphi:L\to L'$ is an isomorphism of geometric lattices and
$T_\varphi e_x=e_{\varphi(x)}$, then
\[
T_\varphi D_{\beta,L}=D_{\beta,L'}T_\varphi,
\qquad
T_\varphi v_k=v_k'.
\]
Moreover, suppose $R$ is diagonal in the flat basis and
\[
\widetilde D=R U^\beta K,\qquad \widetilde Dv_k=kv_{k-1}\quad(1\le k\le r).
\]
Then $Re_x=s(x)^{-1}e_x$ for every $x<\hat1$. In particular, $H_\beta$ depends only on the abstract geometric lattice.
\end{prop}

\begin{proof}
A lattice isomorphism preserves atoms, rank, ordered-basis counts, covers, and the functions $m$ and $s$. Hence it intertwines every term in \eqref{eq:rescaled}, proving naturality. For the normalization statement, write $Re_x=c_xe_x$.  Proposition~\ref{prop:D} gives
\[
U^\beta Kv_k=kSv_{k-1}.
\]
The coefficient of $e_x$, $x\in L_{k-1}$, in $RU^\beta Kv_k$ is therefore
$k c_xs(x)\beta_x$.  Equality with $kv_{k-1}$ forces $c_x=s(x)^{-1}$.
\end{proof}

\begin{remark}\label{rem:canonical}
The word \emph{canonical} refers to this choice-free, isomorphism-natural construction and to the forced diagonal normalization of the canonical adjoint lowering operator $U^\beta K$.  We do not claim that $D_\beta$ is the unique lowering operator on $\R[L]$ extending $d/dt$ on the radial subspace, nor do we claim that it is a derivation of the diamond product.
\end{remark}

Now let $L=L(M)$ for a simple rank-$r$ matroid $M$ with ground set $E=A$.  Since $\Phi_\beta(v_r)=B_M$, the $D_\beta$-free cyclic module is the Macaulay inverse system of the basis-generating polynomial (cf. \cite{MaenoNumata}).  Write
\[
\cc C_k^0(M)=
\Span\{L_{a_k}^\beta\cdots L_{a_1}^\beta v_r:a_1,\ldots,a_k\in E\},
\qquad
h_k^0(M)=\dim\cc C_k^0(M).
\]
The corresponding apolar algebra is Artinian Gorenstein of socle degree $r$; equivalently, its derivative spaces have symmetric Hilbert function, so
$h_k^0=h_{r-k}^0$.

\begin{prop}\label{prop:bounds}
For $0\le k\le r$,
\[
h_k^0(M)\le h_k^\beta(M)\le W_{r-k}(M),
\qquad
h_0^\beta=h_r^\beta=1.
\]
\end{prop}

\begin{proof}
The first inequality comes from the $D_\beta$-free words. Every generator lowers rank by one, hence $\cc C_k^\beta\subseteq\R[L_{r-k}]$, proving the second. Degree zero is $\R v_r$, while $D_\beta^rv_r=r!v_0\ne0$ and $\R[L_0]=\R e_{\hat0}$, proving the endpoint identities. \qedhere
\end{proof}

\section{Examples and Separation from Valuative Invariants}\label{sec:examples}

\begin{prop}[Uniform matroids]\label{prop:uniform}
For every simple uniform matroid $U_{r,n}$,
\[
\cc C^\beta(U_{r,n})=\cc C^0(U_{r,n}),
\qquad
H_{\beta,U_{r,n}}=H^0_{U_{r,n}}.
\]
\end{prop}

\begin{proof}
On the rank-$j$ layer of $U_{r,n}$, $S^{-1}$ acts after lowering by the scalar $(n-j+1)^{-1}$.  Hence
\[
D_\beta e_x=\frac{j}{n-j+1}U^\beta e_x
=\frac{j}{n-j+1}\sum_aL_a^\beta e_x
\]
for all $x\in L_j$.
Every occurrence of $D_\beta$ can therefore be removed from a word without changing the generated cyclic space.
\end{proof}

\begin{example}[The cycle matroid of $K_5$]\label{ex:K5}
Exact computation gives
\[
W=(1,10,25,15,1),\qquad
H^0=1+10q+20q^2+10q^3+q^4,
\]
\[
H_\beta=1+11q+25q^2+10q^3+q^4,
\qquad
\delta=(0,1,5,0,0).
\]
Thus the differential completion already saturates the whole rank-two flat layer in degree two.
\end{example}

We next prove Theorem~\ref{thm:G-separation}.  Bonin constructed rank-three matroids $M,N$ on eight elements having the same $\mathcal G$-invariant but different configurations \cite{BoninG}.  One convenient presentation is by cyclic flats; the rank function is recovered from cyclic-flat data by the formula of Bonin--de Mier \cite{BoninDeMier}.  Both have rank-one cyclic flats $\{0,1\}$ and $\{6,7\}$ and rank-two cyclic flats $\{0,1,2,3\}$ and $\{0,1,6,7\}$; the third rank-two cyclic flat is $\{4,5,6,7\}$ for $M$ and $\{0,1,4,5\}$ for $N$.  Let
\[
P=M^*,\qquad Q=N^*.
\]
Then $P,Q$ are simple rank-five matroids with $\mathcal G(P)=\mathcal G(Q)$ and common Whitney vector
\[
(1,8,28,47,31,1).
\]

\begin{prop}[Exact computer-assisted separation]\label{prop:Bonin}
For this pair,
\[
H_P^0=H_Q^0=1+8q+25q^2+25q^3+8q^4+q^5,
\]
while
\[
H_{\beta,P}=1+9q+30q^2+25q^3+8q^4+q^5,
\]
\[
H_{\beta,Q}=1+9q+28q^2+25q^3+8q^4+q^5.
\]
\end{prop}

\begin{proof}
This is an exact computer-assisted calculation.  Reconstruct the rank functions from the cyclic flats using
\[
r_M(X)=\min_{Z\in\mathcal Z(M)}\bigl(r_M(Z)+|X\setminus Z|\bigr)
\]
and dualize.  The reference script supplied with the source then reconstructs the flat lattices, the ordered-basis weights, $L_a^\beta$, and $D_\beta$ over $\mathbb Q$, and computes the degreewise cyclic ranks by exact rational row reduction.  Independently enumerating all $8!$ rank sequences verifies $\mathcal G(P)=\mathcal G(Q)$.  No floating-point rank decision enters the calculation.  The resulting exact Hilbert vectors are the displayed ones.
\end{proof}

\begin{remark}[The degree-two mechanism behind Theorem~\ref{thm:G-separation}]
The separation in Theorem~\ref{thm:G-separation} already occurs in degree two
and has a concrete description.  Work in the rescaled flat basis
$g_X=\beta_Xe_X$.  For a rank-three flat $X$ and elements $a,b$, define
\[
c_{ab}(X):=
\mathbf 1_{\{\rk(X\cup\{a,b\})=5\}},
\qquad
f_a(X):=\mathbf 1_{\{a\notin X\}}.
\]
Under the identification of the rank-three flat layer with functions on
$L_3$, the classical degree-two vectors are the pair-incidence vectors
$c_{ab}$, while
\[
L_a^\beta D_\beta v_5=5f_a.
\]
For both matroids $P$ and $Q$ in Theorem~\ref{thm:G-separation}, direct
row reduction from their cyclic-flat presentations gives
\[
\dim\Span\{c_{ab}:a,b\in E\}=25.
\]
Modulo this common classical subspace, however, the element-incidence
vectors have different ranks:
\[
\dim\Span\{\overline f_a:a\in E\}
=
\begin{cases}
5,&P,\\
3,&Q.
\end{cases}
\]
The remaining degree-two words $D_\beta L_a^\beta v_5$ and
$D_\beta^2v_5$ lie in the spaces generated by the classical vectors and
these same quotient classes, so they contribute no further directions.
Consequently
\[
h_2^\beta(P)=25+5=30,
\qquad
h_2^\beta(Q)=25+3=28.
\]
Thus the distinction between $P$ and $Q$ is already visible in the
first noncommutative correction to the classical apolar module.  All of
the ranks above can be checked directly from the cyclic-flat
presentations; the ancillary exact implementation provides an
independent rational-arithmetic verification.
\end{remark}

Theorem~\ref{thm:G-separation} follows.  Since every valuative invariant factors through $\mathcal G$ \cite{DerksenFink}, we obtain:

\begin{cor}\label{cor:nonvaluative}
There is no valuative matroid invariant whose restriction to simple matroids is $H_\beta$.  In particular $H_\beta$ is not determined, on simple matroids, by the Tutte polynomial, catenary data, or the Derksen $\mathcal G$-invariant \cite{BoninKung}.
\end{cor}

For the original rank-three pair, however,
\[
H_{\beta,M}=H_{\beta,N}=1+7q+6q^2+q^3,
\]
while their duals are separated.  Thus no transformation depending only on $H_{\beta,M}$ can recover $H_{\beta,M^*}$ in general.

\section{The Generalized Theta Family}\label{sec:theta}

Let $G_t$ be the graph with two vertices joined by four internally disjoint paths of lengths $(1,t,t,t)$, and put $M_t=M(G_t)$.  Then
\[
|E(M_t)|=3t+1,\qquad r(M_t)=3t-2.
\]
The dual $N_t=M_t^*$ has rank three and four parallel classes of sizes $(1,t,t,t)$; collapsing each parallel class to a single element gives $U_{3,4}$.  Complements of flats of $M_t$ are cyclic sets of $N_t$, and the nullity of such a cyclic set is exactly differential degree from the top of $M_t$.  This turns the low-degree flat-coordinate calculation into a finite occupancy problem.

We use the intrinsic formulas \eqref{eq:rescaled}.  Since $M_t$ is simple, for a cover $X\lessdot Y$ one has $m(X,Y)=|Y\setminus X|$.  The normal subgroup $\mathfrak S_t^3$ permutes labels within the three long arms, while an additional $\mathfrak S_3$ permutes the arms; together they form the evident semidirect-product symmetry $\mathfrak S_t^3\rtimes\mathfrak S_3$.  Let $V_i=S^{(t-1,1)}$ and $W_i=S^{(t-2,2)}$ on arm $i$.  We use the standard decomposition of subset-permutation modules into two-row Specht modules \cite{James}.

For the classical part it is useful to make a different dual incidence matrix explicit.  The girth of $M_t$ is $t+1$, so every $k$-subset is independent for $k\le3$.  Let $\mathcal I_k=\binom{E}{k}$ and let
\[
\mathcal S_k=\{S\subseteq E(N_t):\rk_{N_t}(S)=3,\ |S|=k+3\}
\]
be the spanning subsets of $N_t$ of nullity $k$.  If $I\in\mathcal I_k$ and $J$ is a residual monomial occurring in $\del_I B_{M_t}$, then $S=E\setminus J$ belongs to $\mathcal S_k$ and $S\setminus I$ is a basis of $N_t$.  Thus, in the residual-monomial basis indexed equivalently by $S\in\mathcal S_k$, the $k$th squarefree derivatives have coefficient matrix
\begin{equation}\label{eq:theta-classical-incidence}
A_k(S,I)=
\begin{cases}
1,&I\subseteq S\text{ and }S\setminus I\text{ is a basis of }N_t,\\
0,&\text{otherwise}.
\end{cases}
\end{equation}
This spanning-subset model is used only for the classical derivative rank.  The later quotient calculation returns to complements of flats, hence to cyclic subsets of $N_t$.

\begin{lem}[Classical low degrees]\label{lem:theta-core}
For $t\ge6$ and $0\le k\le3$,
\[
h_k^0(M_t)=\binom{3t+1}{k}.
\]
\end{lem}

\begin{proof}
It suffices to show that $A_k$ has full column rank.  Write the occupancy of $I$ as $(\varepsilon,a,b,c)$, where $\varepsilon$ records the singleton parallel class of $N_t$.  If $\varepsilon=0$, select the target occupancy $(0,a+1,b+1,c+1)$; if $\varepsilon=1$, the selected complementing basis uses all three long classes, giving $(1,a+1,b+1,c+1)$.  Ordered first by $\varepsilon$ and then by occupancy, these selected blocks are triangular.  Their diagonal blocks are tensor products of ordinary set-inclusion matrices $W_{j,j+1}(t)$ and identity matrices.  Gottlieb's characteristic-zero rank theorem \cite{Gottlieb} gives full column rank whenever $t\ge2j+1$.  Thus the argument is immediate for $t\ge7$, since $j\le3$.

At $t=6$ the only boundary blocks are the three permutations of source occupancy $(0,3,0,0)$.  On the heavy arm the map to target occupancy $(0,4,1,1)$ contains the inclusion matrix $W_{3,4}(6)$.  By the standard subset-permutation decomposition \cite{James},
\[
\mathbb R\binom{[6]}3
\cong S^{(6)}\oplus S^{(5,1)}\oplus S^{(4,2)}\oplus S^{(3,3)},
\]
and $\ker W_{3,4}(6)=S^{(3,3)}$.  The only competing $\varepsilon=1$ source has heavy-arm factor $W_{2,3}(6)$, whose image is
$S^{(6)}\oplus S^{(5,1)}\oplus S^{(4,2)}$.  Passing also to target occupancy $(1,3,1,1)$ gives the identity on the heavy-arm $3$-subset coordinate, and therefore detects the missing $S^{(3,3)}$ summand.  The combined boundary block is injective.  Hence $A_k$ has full column rank also at $t=6$.
\end{proof}

\begin{lem}[Degree one]\label{lem:theta-degree1}
For $t\ge6$, $D_\beta v_r\in\cc C_1^0(M_t)$, and hence
\[
h_1^\beta(M_t)=3t+1.
\]
\end{lem}

\begin{proof}
In the cyclic-set model, the rank-$(r-1)$ flats are complements of circuits of $N_t$.  Formula \eqref{eq:rescaled} shows that $D_\beta g_{\hat 1}$ is a constant multiple of the constant vector on these circuits.  Give the singleton class weight $-1/2$ and every element in a long class weight $1/2$.  Every two-element circuit inside a long class then has total weight $1$, while every four-element transversal circuit has total weight $-1/2+3/2=1$.  Thus the constant circuit vector is a linear combination of the classical element-incidence vectors $L_a^\beta g_{\hat 1}$.
\end{proof}

\begin{lem}[Degree two]\label{lem:theta-degree2}
For $t\ge6$,
\[
\cc C_2^\beta/\cc C_2^0\cong V_1\oplus V_2\oplus V_3.
\]
\end{lem}

\begin{proof}
Because $D_\beta v_r$ is classical, modulo $\cc C_2^0$ the only new degree-two vectors come from $D_\beta(\cc C_1^0)$.  The trivial target sector is already saturated classically: there are nine source occupancy orbits and nine cyclic target occupancy orbits, and Lemma~\ref{lem:theta-core} makes the induced map on the trivial isotypic component injective, hence bijective.  The element permutation module is $\mathbf 1^{\oplus4}\oplus V_1\oplus V_2\oplus V_3$, so after the trivial sector is removed these are the only possible quotient types.

For $V_1$, compress to the standard alternating line on the first long arm.  Four classical occupancy columns
\[
(0,1,0,1),\ (0,1,1,0),\ (0,2,0,0),\ (1,1,0,0)
\]
together with the $D_\beta(\cc C_1^0)$ column of occupancy $(0,1,0,0)$ have a $5\times5$ minor with determinant
\[
\frac{12}{5}(t-1)\ne0.
\]
Thus one copy of $V_1$ survives modulo the classical image.  The outer $\mathfrak S_3$-symmetry gives $V_2,V_3$, and each occurs with multiplicity at most one in the one-label source, giving the matching upper bound.  The labeled minor is recorded in Appendix~\ref{app:theta-ledger}.
\end{proof}
Thus
\[
h_2^\beta=\binom{3t+1}{2}+3(t-1)=\frac{9t^2+9t-6}{2}.
\]
Degree three requires one further finite calculation.  The one-$D_\beta$ part is
\[
\cc C_3^0+D_\beta(\cc C_2^0)+\sum_aL_a^\beta D_\beta(\cc C_1^0).
\]
(The omitted placement $L_a^\beta L_b^\beta D_\beta v_r$ is classical by Lemma~\ref{lem:theta-degree1}.)  The trivial target sector is again already saturated classically: in degree three there are sixteen source occupancy orbits and sixteen cyclic target occupancy orbits, so Lemma~\ref{lem:theta-core} gives an isomorphism on the trivial isotypic component.

The possible nontrivial types can be read directly from the source permutation modules.  The term $D_\beta(\cc C_2^0)$ is built from two-subset modules; by Young's rule \cite{James} these contain only
\[
V_i,\qquad W_i,\qquad V_i\otimes V_j\quad(i\ne j)
\]
after the trivial constituents are removed.  The ordered two-label source may additionally contain $S^{(t-2,1,1)}$ on a single arm.  The target, however, is a direct sum of tensor products of subset-permutation modules on the three arms, hence contains only two-row Specht modules; the three-row constituent therefore maps to zero.  Thus the displayed three families are the only possible nontrivial quotient types.

For a fixed $V_i$, the target zonal multiplicity is twelve while the classical image has multiplicity nine, so at most three new copies can occur.  For $W_i$, the two possible one-$D_\beta$ source lines have the same image modulo the classical part, as shown explicitly in Appendix~\ref{app:theta-ledger}, so at most one new copy occurs.  For $V_i\otimes V_j$, there is one source line from $D_\beta(\cc C_2^0)$ and two ordered source lines from $L D L$, giving multiplicity at most three.  The compressed witness matrices in Appendix~\ref{app:theta-ledger} have determinants
\begin{equation}\label{eq:theta-dets}
-\frac{41472}{125}(t-1)(3t-4)^2,
\qquad
\frac8{15}(3t-4),
\qquad
\frac6{125}(t-1)^2(3t-4).
\end{equation}
They are nonzero for $t\ge6$, so all three upper bounds are attained.  Therefore
\begin{equation}\label{eq:Q3}
\cc C_3^\beta/\cc C_3^0
\cong
\bigoplus_{i=1}^3 3V_i
\oplus
\bigoplus_{i=1}^3 W_i
\oplus
\bigoplus_{i<j}3(V_i\otimes V_j).
\end{equation}

Words containing two or three $D_\beta$'s add nothing further.  Among words with two $D_\beta$'s, $D_\beta L_a^\beta D_\beta v_r$ lies in the one-$D_\beta$ sector because $D_\beta v_r\in\cc C_1^0$.  Next $D_\beta^2v_r$ is radial, hence lies in the trivial $\mathfrak S_t^3$-isotypic component of the degree-two target; that trivial component was shown in Lemma~\ref{lem:theta-degree2} to be saturated by $\cc C_2^0$.  Therefore $D_\beta^2v_r\in\cc C_2^0$, and consequently $L_a^\beta D_\beta^2v_r\in\cc C_3^0$.  The only remaining pattern is $D_\beta^2L_a^\beta v_r$.  As $a$ varies, these vectors form a quotient of the element permutation module, so only trivial constituents and the $V_i$ can occur.  The degree-three trivial sector is classical and the full available $V_i$ target multiplicity is already saturated in \eqref{eq:Q3}.  Finally $D_\beta^3v_r$ is radial, hence trivial-isotypic, and the degree-three trivial sector is classically saturated.  Thus no higher-$D_\beta$ word enlarges \eqref{eq:Q3}.

\begin{proof}[Proof of Theorem~\ref{thm:theta}]
The degree-zero coefficient is one, and Lemma~\ref{lem:theta-degree1} gives $h_1^\beta=3t+1$.  Lemma~\ref{lem:theta-degree2} and Lemma~\ref{lem:theta-core} give
\[
h_2^\beta=\binom{3t+1}{2}+3(t-1)=\frac{9t^2+9t-6}{2}.
\]
From \eqref{eq:Q3} and
\[
\dim V_i=t-1,\qquad \dim W_i=\frac{t(t-3)}2,
\]
we obtain
\[
\dim(\cc C_3^\beta/\cc C_3^0)
=9(t-1)+\frac{3t(t-3)}2+9(t-1)^2
=\frac{3t(7t-9)}2.
\]
Since $h_3^0=\binom{3t+1}{3}$,
\[
h_3^\beta=\frac{9t^3+21t^2-28t}{2}.
\]
Substitution gives the stated positive log-concavity gap.
\end{proof}

For Larson's $t=26$ example, Whitney log-concavity fails at rank $r-2$ \cite{Larson}; ranks $r-1,r-2,r-3$ correspond to differential degrees $1,2,3$.  Theorem~\ref{thm:theta} therefore proves the opposite strict inequality for $H_\beta$ at precisely the corresponding index.  The determinant computation in \eqref{eq:theta-dets} is recorded explicitly in Appendix~\ref{app:theta-ledger}.

\section{Computational Evidence}\label{sec:computation}

The exact reference implementation takes only a finite ground set and rank oracle.  It reconstructs the flat lattice, the $\beta_F$, the operators $L_a^\beta,D_\beta$, and the cyclic spaces by exact rational row reduction.  It independently checks
\[
h_k^0\le h_k^\beta\le W_{r-k},\qquad h_0^\beta=h_r^\beta=1,
\]
and records any failure of the conjectures.

\begin{prop}[Complete eight-element census]\label{prop:n8}
For every one of the $950$ pairwise nonisomorphic simple matroids on eight elements in the Mayhew--Royle catalogue \cite{MayhewRoyle}, $H_\beta$ is strictly log-concave, the defect is unimodal, and front-loading holds. Moreover $H_\beta\ne H^0$ in $894$ cases.
\end{prop}

\begin{proof}
The reference implementation computes all cyclic-space ranks over $\Q$ by exact row reduction. Running it over the complete simple eight-element census gives the stated counts; the ancillary output records the degreewise Hilbert vectors and all assertion checks.
\end{proof}

The rank distribution of the census is
\[
\begin{array}{c|rrrrrrr}
r&2&3&4&5&6&7&8\\\hline
\#&1&68&617&217&40&6&1.
\end{array}
\]
The smallest exact log-concavity gap is $35$. Among the $1220$ nontrivial front-loading comparisons, $1107$ are strict.  The largest defect observed is
\[
H^0=(1,8,24,34,24,8,1),\qquad
H_\beta=(1,8,28,46,28,8,1).
\]
Thus the evidence is not driven by cases where $D_\beta$ does nothing.

\begin{prop}[Complete nine-element census]\label{prop:n9}
For all $376{,}467$ pairwise nonisomorphic simple matroids on nine elements in the Mayhew--Royle catalogue \cite{MayhewRoyle}, $H_\beta$ is strictly log-concave and the defect is unimodal. The only failure of front-loading is \texttt{M195313}, the example of Guo, Mo, Li, and Deng.
\end{prop}

\begin{proof}
The census reconstructs flats and covers from each catalogue rank function. It uses \eqref{eq:rescaled} and propagates column spaces by exact elimination over $\Q$. The rank counts for ranks $2,\ldots,9$ are
\[
(1,383,185981,188936,1092,66,7,1).
\]
A separate pass matches the catalogue identifiers of the export and recomputes the inequalities from the resulting Hilbert vectors. Four selected profiles, including \texttt{M195313}, were also recomputed by the reference implementation. The ancillary census code and its analysis record the checks.
\end{proof}

The smallest adjacent log-concavity gap $(h_k^\beta)^2-h_{k-1}^\beta h_{k+1}^\beta$ is $39$. The exact census supersedes the earlier modular sample of $400$ nine-element matroids, which missed \texttt{M195313}.

\begin{example}[The defect can fail log-concavity]\label{ex:defect-lc}
For the simple rank-six catalogue matroid \texttt{M384298}, exact computation gives
\[
h^0=(1,9,32,58,32,9,1),\qquad h^\beta=(1,10,42,62,34,9,1).
\]
The defect $(0,1,10,4,2,0,0)$ is unimodal and satisfies $\delta_3^2=16<20=\delta_2\delta_4$.
\end{example}

\section{Questions}\label{sec:questions}

\begin{question}
Does the $D_\beta$-filtration have an associated graded object carrying a Lorentzian or Lefschetz structure that could imply Conjecture~\ref{conj:lc}, perhaps in the sense suggested by recent structural results on Lefschetz modules \cite{AminiHuhLarson}?
\end{question}

\begin{question}
Do the defect spaces $\cc C_k^\beta(M)/\cc C_k^0(M)$ admit maps between consecutive degrees that explain their conjectural unimodality?
\end{question}

\begin{question}
How does the filtered cyclic module behave under direct sum, deletion, contraction, and duality?  Simple formulas for $H_\beta$ itself fail, and the Bonin pair shows that duality is not determined by $H_\beta$ alone.
\end{question}

\section*{Methods}

OpenAI's ChatGPT 5.5 and ChatGPT 5.6 Sol and Anthropic's Claude Opus 5 were used at exploratory stages to generate examples and computations and to assist in drafting the manuscript. In addition, ChatGPT 5.6 Sol generated the code included with the ancillary files and carried out the computations for the generalized theta family. The conceptual framework is the sole work of the author. The author has thoroughly edited the manuscript, has independently verified the code and computations, and takes sole responsibility for the accuracy of all results.

\appendix

\section{Theta determinant witnesses}\label{app:theta-ledger}

We record the finite minors behind Lemma~\ref{lem:theta-degree2} and \eqref{eq:theta-dets}.  For $S^{(t-j,j)}$ on a long arm, choose the standard zonal vector obtained by alternating on $j$ fixed disjoint pairs and symmetrizing over the remaining labels.  Tensor these vectors across the three long arms.  Equivariance reduces the incidence maps to multiplicity matrices indexed by occupancy types.

\begin{lem}[Affine occupancy entries]\label{lem:affine-occupancy}
Every entry in the compressed degree-two and degree-three witness matrices below is affine in $t$.
\end{lem}

\begin{proof}
After the distinguished pairs used in the zonal alternation are fixed, each degree-$\le3$ source/target incidence condition leaves at most one undistinguished label in a long arm free.  Its contribution is therefore either constant or a constant multiple of $t-c$ for a fixed integer $c$.  Tensoring with the fixed incidences in the other arms preserves affine dependence.  For example, the degree-two $V_1$ entry from source occupancy $(0,1,0,0)$ to target occupancy $(0,2,0,2)$ has one distinguished first-arm choice and $t-1$ undistinguished choices, giving the factor $\frac32(t-1)$ after the fixed normalization.
\end{proof}

For degree two in the $V_1$ sector, use the target rows
\[
(0,2,0,2),\ (0,2,2,0),\ (0,3,0,0),\ (1,1,1,2),\ (1,2,1,1)
\]
and the four classical source columns
\[
(0,1,0,1),\ (0,1,1,0),\ (0,2,0,0),\ (1,1,0,0),
\]
followed by the $D_\beta(\cc C_1^0)$ source $(0,1,0,0)$.  The compressed matrix is
\[
\begin{pmatrix}
2&0&0&0&\frac32(t-1)\\
0&2&0&0&\frac32(t-1)\\
0&0&2&0&2(t-1)\\
2&0&0&0&\frac65(t-1)\\
1&1&1&1&\frac{12}{5}(t-1)
\end{pmatrix},
\]
whose determinant is $12(t-1)/5$.

For degree three in the $V_1$ sector, the twelve target zonal rows, in order, are
\[
\begin{gathered}
(0,2,0,3),(0,2,2,2),(0,2,3,0),(0,3,0,2),(0,3,2,0),(0,4,0,0),\\
(1,1,1,3),(1,1,2,2),(1,1,3,1),(1,2,1,2),(1,2,2,1),(1,3,1,1).
\end{gathered}
\]
The first nine columns are the classical source occupancies
\[
\begin{gathered}
(0,1,0,2),(0,1,1,1),(0,1,2,0),(0,2,0,1),(0,2,1,0),\\
(0,3,0,0),(1,1,0,1),(1,1,1,0),(1,2,0,0),
\end{gathered}
\]
followed by the two $D_\beta(\cc C_2^0)$ columns $(0,1,0,1),(0,1,1,0)$ and the $L D L$ column with outer label in arm $2$ and inner label in arm $1$.  In these row and column orders the resulting matrix is
{\scriptsize
\[
\begin{pmatrix}
3&0&0&0&0&0&0&0&0&\frac{6(3t-4)}5&0&0\\
0&4&0&0&0&0&0&0&0&\frac{2(3t-4)}3&\frac{2(3t-4)}3&3(t-1)\\
0&0&3&0&0&0&0&0&0&0&\frac{6(3t-4)}5&\frac{9(t-1)}2\\
0&0&0&4&0&0&0&0&0&\frac{4(3t-4)}5&0&0\\
0&0&0&0&4&0&0&0&0&0&\frac{4(3t-4)}5&4(t-1)\\
0&0&0&0&0&3&0&0&0&0&0&0\\
3&0&0&0&0&0&0&0&0&3t-4&0&0\\
0&4&0&0&0&0&0&0&0&\frac{2(3t-4)}3&\frac{2(3t-4)}3&\frac{12(t-1)}5\\
0&0&3&0&0&0&0&0&0&0&3t-4&\frac{18(t-1)}5\\
1&2&0&2&0&0&2&0&0&\frac{4(3t-4)}3&\frac{3t-4}3&\frac{3(t-1)}2\\
0&2&1&0&2&0&0&2&0&\frac{3t-4}3&\frac{4(3t-4)}3&\frac{24(t-1)}5\\
0&0&0&2&2&1&0&0&2&\frac{3t-4}3&\frac{3t-4}3&2(t-1)
\end{pmatrix}.
\]
}
Its determinant is
\[
-\frac{41472}{125}(t-1)(3t-4)^2.
\]
Thus the $V_1$ target multiplicity is saturated; arm symmetry gives the same conclusion for $V_2,V_3$.

For $W_1$, take target rows
\[
(0,3,0,2),(0,3,2,0),(0,4,0,0),(1,2,1,2),(1,3,1,1)
\]
and classical columns
\[
(0,2,0,1),(0,2,1,0),(0,3,0,0),(1,2,0,0),
\]
followed by the $D_\beta(\cc C_2^0)$ source $(0,2,0,0)$.  The witness matrix is
\[
\begin{pmatrix}
2&0&0&0&\frac{6t-8}{5}\\
0&2&0&0&\frac{6t-8}{5}\\
0&0&2&0&\frac{3t-4}{2}\\
2&0&0&0&t-\frac43\\
1&1&1&1&2t-\frac83
\end{pmatrix},
\]
with determinant
\[
\frac8{15}(3t-4).
\]
If $C_1,\ldots,C_4$ denote the four classical columns, $D$ the final displayed column, and $R$ the additional same-arm $L D L$ zonal column, direct orbit counting gives
\[
R=-\frac{3(t-1)}{10}(C_1+C_2)-\frac{t-1}{4}C_3
-\frac{7(t-1)}{20}C_4+\frac{9(t-1)}{3t-4}D.
\]
Thus the two possible $W_1$ source lines have the same image modulo the classical part, and the $W_1$ quotient multiplicity is exactly one.
For $V_1\otimes V_2$, take target rows
\[
(0,2,2,2),(0,2,3,0),(0,3,2,0),(1,1,2,2),(1,1,3,1),(1,2,1,2),(1,2,2,1)
\]
and classical columns
\[
(0,1,1,1),(0,1,2,0),(0,2,1,0),(1,1,1,0),
\]
followed by three nonclassical columns: the $D_\beta(\cc C_2^0)$ source $(0,1,1,0)$ and the two ordered $L D L$ columns with outer/inner arms
\[
(1,2)\qquad\text{and}\qquad(2,1).
\]
The witness matrix is
{\scriptsize
\[
\begin{pmatrix}
2&0&0&0&t-\frac43&\frac32(t-1)&\frac32(t-1)\\
0&2&0&0&\frac{6t-8}{5}&2(t-1)&\frac32(t-1)\\
0&0&2&0&\frac{6t-8}{5}&\frac32(t-1)&2(t-1)\\
2&0&0&0&t-\frac43&\frac32(t-1)&\frac65(t-1)\\
0&2&0&0&t-\frac43&2(t-1)&\frac65(t-1)\\
2&0&0&0&t-\frac43&\frac65(t-1)&\frac32(t-1)\\
1&1&1&1&2t-\frac83&\frac{12}{5}(t-1)&\frac{12}{5}(t-1)
\end{pmatrix}.
\]
}
Its determinant is
\[
\frac6{125}(t-1)^2(3t-4).
\]
There is one $D_\beta(\cc C_2^0)$ line and two ordered $L D L$ lines in this sector, so the quotient multiplicity is at most three; the nonzero determinant attains this bound.

By Lemma~\ref{lem:affine-occupancy}, all displayed entries are affine in $t$ and can be checked by direct orbit counting.  The companion SymPy audit reconstructs the matrices from the rank-three parallel-class dual and independently verifies the symbolic formulas by exact arithmetic.

\section{Reference computation}\label{app:algorithm}

The source archive includes the exact reference script \texttt{hbeta\_oracle\_audit.py} and the theta determinant ledger \texttt{theta\_h3\_determinant\_ledger.py}.  For reproducibility, the exact implementation uses only a ground set and a matroid rank oracle.  It performs the following steps.
\begin{enumerate}
\item Enumerate flats using the closure test $e\in\cl(X)$ iff $r(X\cup\{e\})=r(X)$.
\item Compute $\beta_F$ for every flat by counting ordered atom bases (equivalently, $\beta_F=\rk(F)!$ times the number of bases of $M|F$).
\item Build the exact rational matrices $L_a^\beta$ from Proposition~\ref{prop:adjoint-polynomial} and form $D_\beta=S^{-1}(\sum_aL_a^\beta)K$.
\item Starting from $\cc C_0^\beta=\R v_r$, propagate recursively by
\[
\cc C_{k+1}^\beta
=
\Span\Bigl(
D_\beta\cc C_k^\beta\cup
\bigcup_{a\in A}L_a^\beta\cc C_k^\beta
\Bigr),
\qquad 0\le k<r,
\]
replacing each span by a column basis after every step.  This is exhaustive: every generator lowers lattice rank by exactly one, so $\cc C_k^\beta\subseteq\R[L_{r-k}]$ and the process stops after degree $r$.  Thus the program never enumerates all noncommutative words individually.  Omitting $D_\beta$ in the same recursion gives $H^0$.
\item Assert the radial identity and the bounds in Proposition~\ref{prop:bounds}, then check log-concavity, defect unimodality, and front-loading.
\end{enumerate}
The nine-element census uses the rescaled formulas \eqref{eq:rescaled} over $\Q$. Its driver is \path{anc/census/census.py} and the coefficient checks can be repeated with \path{anc/analysis/analyze.py}. The ancillary files also retain the original modular sample and its validation report.

\end{document}